\documentclass{amsart}
\usepackage{amsfonts, amsbsy, amscd,color,amsmath,amsopn,amssymb,enumerate}

\newtheorem{thm}{Theorem}[section]

\newtheorem{fact}[thm]{Fact}

\newtheorem{pro}[thm]{Proposition}

\newtheorem{rem}[thm]{Remark}
\newtheorem{que}[thm]{Question}

\theoremstyle{definition}
\newtheorem{defi}[thm]{Definition}
\newtheorem*{acknowledgement}{Acknowledgement}

\newcommand{\M}{\mathcal{M}}

\newcommand{\F}{\mathcal{F}}
\newcommand{\U}{\mathcal{U}}
\newcommand{\V}{\mathcal{V}}
\newcommand{\G}{\mathcal{G}}
\newcommand{\B}{\mathcal{B}}
\newcommand{\C}{\mathcal{C}}
\newcommand{\Power}{\mathcal{P}}
\newcommand{\N}{\mathbb{N}}

\newcommand{\edowod}{\hfill$\square$}

\begin{document}
\title{THE DIMENSION OF HYPERSPACES OF NON-METRIZABLE CONTINUA}
\author{Wojciech Stadnicki}
\address{Mathematical Institute\\University of Wroc\l aw\\
Pl. Grunwaldzki 2/4, 50--384 Wroc\l aw\\Poland}
\email{stadnicki@math.uni.wroc.pl}
\date{\today}

\begin{abstract}
We prove that, for any Hausdorff continuum $X$, if $\dim X \ge 2$ then the
hyperspace $C(X)$ of subcontinua of $X$ is not a $C$-space;
if $\dim X=1$ and $X$ is hereditarily indecomposable then $\dim C(X)=2$
or $C(X)$ is not a $C$-space.
This generalizes results known for metric continua. 
\end{abstract}
\maketitle

\section{Introduction}
Throughout the paper all spaces are normal.
A continuum is a compact, connected Hausdorff space.
By dimension we always mean the covering dimension $\dim$.
A continuum $X$ is hereditarily indecomposable iff for each subcontinua
$A, B\subseteq X$ we have $A\subseteq B$, $B\subseteq A$ or $A\cap B=\emptyset$.
For a compact $X$ denote by $K(X)$ the hyperspace
 of all non-empty subcompacta of $X$,
equipped with the Vietoris topology.
By $C(X)$ we denote the hyperspace of all non-empty subcontinua of $X$,
with the topology inherited from $K(X)$.

\begin{defi}
A space $X$ is a $C$-space (or has property $C$) if and only if
for each sequence $\U_1, \U_2, \ldots$
of open covers of $X$, there exists a sequence $\V_1, \V_2, \ldots$,
such that each $\V_i$ is a family of pairwise disjoint open subsets of $X$,
$\V_i\prec\U_i$ ($\V_i$ refines $\U_i$, i.e. $\forall~V\in\V_i~\exists U\in\U_i~V\subseteq U$) and
$\bigcup_{i=1}^\infty \V_i$ is a cover of $X$. 
\end{defi}
We refer to \cite{EngD} for basic properties of $C$-spaces.
It is easy to observe that $C$-spaces are weakly infinite dimensional.
The class of $C$-spaces contains finite dimensional spaces
and countable dimensional metric spaces.

We prove the following theorem:
\begin{thm}
\label{non-metric}~
\begin{enumerate}[$(i)$]
\item
Suppose $X$ is a continuum of dimension $\geq 2$. Then $C(X)$ is not a $C$-space.
\item Suppose $X$ is a 1-dimensional hereditarily indecomposable continuum.
Then either $\dim C(X)=2$ or $C(X)$ is not a $C$-space.
\end{enumerate}
\end{thm}

The theorem is already known for metric continua.
Part $(i)$ was stated by M. Levin and J. T. Rogers, Jr.  in \cite{LeRo}.
Part $(ii)$ can be obtained using methods from \cite{EbNa, LeRo, LeSt} (see \cite[Theorem 3.1]{Sta}).

To prove it for non-metric spaces we use the technique
of lattices and Wallman representations
as well as some set-theoretical methods,
as it was done in \cite{BHHS}.
We refer to \cite{vdS} for the definition of a lattice and preliminary facts on Wallman spaces.
We consider only distributive and separative lattices.

\section{Lattices and Wallman spaces}

For a compact space $X$ we consider the lattice $2^X$ of closed subsets of $X$ with
$\cup$ and $\cap$ as lattice operations,
$\emptyset$ and $X$ as the minimal and maximal elements.
Each lattice $L$ corresponds to the Wallman space $wL$ consisting of all
ultrafilters on $L$. For $a\in L $ let $\widehat{a}=\lbrace u\in wL\colon a\in u\rbrace$. We define
the topology in $wL$ taking the family $\lbrace\widehat{a}\colon a\in L\rbrace$ as a base for closed sets.

It is easy to show that $w2^X$ is homeomorphic to $X$.
More generally, the following fact holds true:
\begin{fact}
\label{homeofact}
If $\F$ is a base for closed sets in $X$ which is closed under finite unions and intersections
(so $\F$ is a lattice),
then $w\F$ is homeomorphic to $X$.
\end{fact}
\begin{proof}
We define the homeomorphism $h\colon X\to w\F$ in the natural way:
$h(x)=\lbrace F\in\F\colon x\in F\rbrace$. It is not difficult but tedious to verify
that $h$ is a well-defined homeomorphism indeed. We leave it as an exercise.
\end{proof}
\begin{defi} A lattice $L$ is normal iff
$$L\models\forall a, b\left(a\cap b=0_L \rightarrow \exists c, d \left(
c\cup d=1_L \wedge c\cap a=0_L
\wedge d\cap b=0_L\right)\right)$$
\end{defi}
We collect some well-known observations.
\begin{fact}[see, e.g., \cite{vdS}]
\label{t2}
$L$ is normal if and only if $wL$ is Hausdorff.
\end{fact}
\begin{fact}[{\cite[Theorem 2.6]{vdS}}]
\label{compactmetriclattice}
If $L$ is a countable normal lattice then $wL$ is a compact metric space.
\end{fact}
\begin{rem}
\label{surject}
A sublattice $L$ of $L^*$ yields the continuous surjection \mbox{$q\colon wL^*\to wL$},
 given by $q(u)=u\cap L$.
\end{rem}

\section{Proof of Theorem \ref{non-metric}}
The proof is rather simple, but it uses some set-theoretic framework.
We deal with some inner model of (large enough fragment of)
ZFC and its countable elementary submodel.

Our strategy is to bring the non-metric case to the metric one.
Suppose $X$ is a non-metric continuum.
We will find a countable sublattice $L\subseteq 2^X$ such that
$wL$ is a metric continuum, 
$\dim wL=\dim X$ and
$\dim C(wL)=\dim C(X)$.
Moreover, $wL$ [$C(wL)$] is hereditarily indecomposable if and only if such is $X$ [$C(X)$] and
 $wL$ [$C(wL)$] is a $C$-space if and only if such is $X$ [$C(X)$].

We apply the technique used in \cite{BHHS} to find the sublattice $L$.
 
For an infinite cardinal $\kappa$, $H(\kappa)$ is the set of all sets $x$, such that $|TC(x)|<\kappa$.
($TC$ is the transitive closure, i.e. $TC(x)=x\cup\bigcup x\cup\bigcup\bigcup x\cup\ldots$).
If $\kappa$ is regular then $H(\kappa)$ is a model of ZFC without the Power Set Axiom
(see \cite[p. 162]{Jech}). But if $\kappa$ is large enough, then there
are power sets in $H(\kappa)$ for all sets we need.

Let $X$ be a (non-metric) continuum. Fix a suitably large regular cardinal $\kappa$
(it is enough if $\Power(\Power(X))\in H(\kappa)$).
Take a countable elementary submodel \mbox{$\M\prec H(\kappa)$}, such that $X\in\M$
(use the L\"owenheim-Skolem theorem).
Then $\M$ also models enough of ZFC. Moreover, every finite subset
of $\M$ belongs to $\M$.
Denote $L=2^X\cap\M$. By elementarity, $L$ is a normal sublattice of $2^X$.
Since $L$ is countable, applying Fact \ref{compactmetriclattice} and Remark \ref{surject}, we obtain:
\begin{fact}
\label{wL-metric}
$wL$ is a metric continuum.
\edowod
\end{fact}

Let us recall two well-known facts.
\begin{pro}[see {\cite[Subsection 4.1]{Hart}}]
\label{refl-dim}
$\dim X=\dim wL$. More generally, let $K^*$ be a lattice in
$\M$ and $K=K^*\cap\M$. Then $\dim wK^*=\dim wK$.
\end{pro}

\begin{pro}
\label{HI-refl}
A continuum $X$ is hereditarily indecomposable if and only if such is $wL$.
\end{pro}
The \emph{if} part is straightforward. For the 
\emph{only if} see \mbox{\cite[Lemma 2.2]{HartEPol}}.

Now we prove a similar fact about property $C$.
\begin{thm}
\label{relf-C}
The space $X$ is a $C$-space if and only if so is $wL$.
More generally, let $K^*$ be a lattice in
$\M$ and $K=K^*\cap\M$. Then $wK^*$ is a $C$-space if and only if such is $wK$.
\end{thm}
\begin{proof}
We provide the proof for the first part of the proposition.
It can be easily adopted for the more general statement.

Denote
$\B=\lbrace wL\setminus \widehat{F}\colon\ F\in L\rbrace$ (the open base for $wL$, 
which is closed under finite unions and intersections).

~

$(\Leftarrow)$
We will show that if $X$ is not a $C$-space then neither is $wL$.
Assume X is not a C-space. Then, by
compactness there exists
a sequence $(\U_i)_{i=1}^\infty$ of finite
open covers of $X$, such that for every $m\geq1$ and finite families of open disjoint sets
$\V_1, \V_2,\ldots, \V_m$ which satisfy $\V_i\prec\U_i$,
their union $\V_1\cup \V_2\cup\ldots\cup\V_m$ is not a cover of $X$
(compactness allows to consider only finite families).
Translating it into terms of lattice $2^X$ we obtain that $H(\kappa)$
models the following sentence $\varphi$:\\\\
\begin{tabular}{ll}
$(\varphi)$ &
$
\hspace{-0.25cm}\left\{\hspace{0.2cm}
\begin{minipage}{0.91\textwidth}
There exists a sequence $(\F_i)_{i=1}^\infty$ of finite subsets of $2^X$ such that
for each $i\geq 1$ the intersection $\bigcap\F_i$ is empty and
for every $m\geq1$ and finite $\G_1,\G_2,\ldots,\G_m\subseteq2^X$ the following holds:\\\\
\begin{tabular}{ll}
$(\ast)$ &
\begin{minipage}{0.9\textwidth}
If for each
$j\leq m$ and $G\in\G_j$ there exists $F\in\F_j$ such that $F\subseteq G$ and 
for any distinct $G, G'\in\G_j$ we have $G\cup G'=X$, then
$\bigcap(\G_1\cup\G_2\cup\ldots\cup\G_m)\neq\emptyset$.
\end{minipage}
\end{tabular}
\end{minipage}\right.$
\end{tabular}\\\\
$\M\models\varphi$ by elementarity. So there is $(\F_i)_{i=1}^\infty\in\M$
as in $\varphi$, such that $(\ast)$ holds for every $m<\omega$ and $\G_1, \G_2, \ldots, \G_m\in\M$.

The sequence $(\F_i)_{i=1}^\infty$ gives rise to a sequence $(\U_i)_{i=1}^\infty$ of open covers of $wL$
(namely $\U_i=\lbrace wL\setminus \widehat{F}\colon F\in\F_i\rbrace$),
which witnesses that $wL$ is not a $C$-space.
Indeed, suppose we have a finite sequence
$\V_1, \V_2, \ldots, \V_m$ of finite families of open disjoint sets,
$\V_i\prec\U_i$ and their union is a cover of $wL$.
We can produce $\V'_1, \V'_2, \ldots, \V'_m$, which are additionally contained in the base $\B$:
shrink each $V\in\bigcup_{i=1}^m\V_i$ to a closed set $C_V$ so that
$\bigcup_{i=1}^m\{C_V\colon V\in\V_i\}$ forms a closed cover of $wL$.
Since $C_V$ is compact, it can be covered by finitely many sets $B^V_1, B^V_2,\ldots,B^V_{j(V)}\subseteq V$
from the basis $\B$. Let $V'=B^V_1\cup B^V_2\cup\ldots\cup B^V_{j(V)}$.
We have $V'\in\B$, since $\B$ is closed under finite unions.
Define $\V'_i=\{V'\colon V\in\V_i\}$.

Having $\V'_1, \V'_2, \ldots, \V'_m$ it is easy to get
$\G_1, \G_2, \ldots, \G_m\in\M$ which do not satisfy $(\ast)$.
Indeed, each $V'\in\V'_i$ is given by some $F_{V'}\in L$ via
$V'=wL\setminus\widehat F_{V'}$. Then $\G_i=\{F_{V'}\colon V'\in\V_i'\}$.
Since $\V_i'\subseteq\B$, we have $\G_i\in\M$.

~

$(\Rightarrow)$
Suppose $\U_1, \U_2,\ldots$ is a sequence of finite open covers of $wL$, 
say $\U_i=\{U_{i 1}, U_{i 2},\ldots, U_{i k_i}\}$.
 Without loss of generality
we may assume that each $\U_i$ consists of sets from $\B$, i.e. for each $i\in\N$ and $j\leq k_i$
there is some $F_{i j}\in\M$ closed in $X$ such that $U_{i j}=wL\setminus\widehat{F_{i j}}$.

Define $U_{i j}'=X\setminus F_{i j}$ and $\U_i'=\{U_{i 1}', U_{i 2}',\ldots, U_{i k_i}'\}$.
Note that $\U'_i$ is an open cover of $X$ since
$F_{i 1}\cap F_{i 2}\cap\ldots\cap F_{i k_i}=\emptyset$ 
($\U_i$ is a cover of $wL$).

Since $X$ is a compact $C$-space there exist $n\in\N$ and finite families of pairwise disjoint
open sets $\V'_1, \V'_2,\ldots,\V'_n$ such that each $\V'_i$ refines $\U'_i$ and $\bigcup_{i=1}^n\V'_i$
is a cover of $X$. Let us code this in terms of the lattice $2^X$.
First denote $\V'_i=\{V'_{i 1}, V'_{i 2},\ldots, V'_{i l_i}\}$ and
$G'_{i j}=X\setminus V'_{i j}$ for $i\leq n$ and $j\leq l_i$.
The following sentence $\psi$ is true in $H(\kappa)$:
\begin{tabular}{ll}
$(\psi)$ &
$
\hspace{-0.25cm}\left\{\hspace{0.2cm}
\begin{minipage}{0.91\textwidth}
There exist $G'_{1 1}, G'_{1 2},\ldots,G'_{1 l_1}, G'_{2 1}, G'_{2 2},\ldots G'_{2 l_2},\ldots,
G'_{n 1}, G'_{n 2},\ldots,G'_{n l_n}$ such that:\\
$(1)~\bigwedge_{i=1}^n \left(\bigwedge_{1\leq j<j'\leq l_i}
\left( G'_{i j}\cup G'_{i j'}=X\right)\right)$\\
$(2)~\bigwedge_{i=1}^n \left(\bigwedge_{j=1}^{l_i}\left(
\bigvee_{j'=1}^{k_i}\left( G'_{i j}\cap F_{i j'}=F_{i j'}\right)\right)\right)$\\
$(3)~\bigcap_{i=1}^n\bigcap_{j=1}^{l_i} G'_{i j}=\emptyset$.
\end{minipage}\right.$
\end{tabular}\\\\
Symbols $\bigwedge$ and $\bigvee$ abbreviate finite conjuctions and disjunctions.
Note that $F_{i j}$'s appear in $\psi$ as parameters from $\M$.

We have $H(\kappa)\models\psi$ and by elementarity $\M\models\psi$.
Hence, for $i\leq n$ and $j\leq l_i$  there are $G_{i j}\in\M$
which satisfy $(1-3)$ when placed in $\psi$ instead of $G'_{i j}$.
Take $V_{i j}=wL\setminus\widehat{G_{i j}}$ and $\V_i=\{V_{i 1}, V_{i 2},\ldots,V_{i l_k}\}$.
Then $\V_1, \V_2,\ldots,\V_n$ are families of pairwise disjoint sets (by $(1)$),
open in $wL$. For $i\leq n$ the
family $\V_i$ refines $\U_i$ (by $(2)$) and $\bigcup_{i=1}^n\V_i$ is a cover of $wL$ (by $(3)$).
\end{proof}

Now we will link the space $X$ with its hyperspace $C(X)$ in terms of lattices.
Namely, having the lattice
$2^X$ we define a lattice $K^*\in\M$, such that $wK^*$ is homeomorphic to $C(X)$.
Then, taking $K=K^*\cap\M$ we will show that $wK$ is homeomorphic to $C(wL)$.

$C(X)$ is defined (as a set) only in terms of $2^X$:
$$C(X)=\lbrace F\in 2^X\colon\neg(\exists G_1, G_2\in2^X)(G_1\cup G_2=F\wedge G_1\cap G_2=\emptyset)\rbrace.$$
Define $K^*$ as a sublattice of $(\Power(C(X)), \cup, \cap, \emptyset, C(X))$
generated by the family $\lbrace\F^*\colon\F\in[2^X]^{<\omega}\rbrace$, where
$$\F^*=C(X)\setminus\lbrace G\in C(X)\colon G\cap\bigcap
\F=\emptyset\wedge(\forall F\in\F)(F\cup G\neq F)\rbrace.$$

The lattice $K^*$ is the closure under finite unions and intersections
of the family of sets $\F^*$ for all finite $\F\subseteq2^X$.
It is easy to verify that sets $\F^*$ form a closed base for $C(X)$.
Hence, $K^*$ is a closed base and a lattice simultaneously.
By Fact \ref{homeofact}, we get:

\begin{rem}
\label{obser}
 $C(X)$ is homeomorphic to $wK^*$.
\end{rem}
Since $X\in\M$, it follows directly by the definition of $K^*$ that $K^*\in\M$.
Take $K=K^*\cap\M$. The only thing we still lack is:
\begin{pro}
\label{kluczowa-propozycja}
$wK$ is homeomorphic to $C(wL)$.
\end{pro}
\begin{proof}
We know that $K^*$ is generated by the family $\{\F^*\colon\F\in[2^X]^{<\omega}\}$.
By elementarity, $K$ is generated by $\{\F^*\colon\F\in[L]^{<\omega}\}$.
Note that a basic closed set in $C(wL)$ is determined by $\F\in[L]^{<\omega}$
via the formula 
$$\C_\F=C(wL)\setminus\lbrace C\in C(wL)\colon C\cap\bigcap
\lbrace\widehat{F}\colon F\in\F\rbrace=\emptyset\wedge
(\forall F\in\F)(\widehat{F}\cup C\neq \widehat{F})\rbrace$$
(since $L$ is isomorphic to a closed base for $wL$).
Hence, the lattice $K$ is isomorphic to the lattice generated by $\{\C_\F\colon\F\in[L]^{<\omega}\}$,
which forms a closed base for $C(wL)$. By Fact \ref{homeofact} $wK$ is homeomorphic to $C(wL)$.
\end{proof}

Now we have all ingredients to prove Theorem \ref{non-metric}.
\begin{proof}[Proof of Theorem \ref{non-metric}]~

$(i)$
Suppose that $\dim X\geq2$. Proposition \ref{refl-dim} gives $\dim wL\geq2$.
By the results of M. Levin, J. T. Rogers, Jr. for metric continua \cite{LeRo}
we have that $C(wL)$ is not a $C$-space.
But $C(wL)$ is homeomorphic to $wK$ (Proposition \ref{kluczowa-propozycja}).
Hence $wK^*$ is neither a $C$-space (Theorem \ref{relf-C}).
By Remark \ref{obser}, $C(X)$ is homeomorphic to $wK^*$, so it is not a $C$-space.

$(ii)$ Similarily, suppose that $X$ is a $1$-dimensional, hereditarily indecomposable continuum.
Then $wL$ is also $1$-dimensional (Proposition \ref{refl-dim}) and hereditarily indecomposable
(Proposition \ref{HI-refl}). By result known for metric continua (\cite[Theorem 3.1]{Sta})
we have that $C(wL)$ is either $2$-dimensional or is not a $C$-space.
By Proposition \ref{kluczowa-propozycja}, $C(wL)$ is homeomorphic to $wK$.
Therefore, $wK^*$ is either $2$-dimensional (Proposition \ref{refl-dim})
or is not a $C$-space (Theorem \ref{relf-C}).
But $wK^*$ is homeomorphic to $C(X)$ by Remark \ref{obser}.
\end{proof}

\section{Remarks on $m$-$C$-spaces}
\begin{defi}[\cite{Fed}]
For $m\geq 2$ a space $X$ is said to be an $m$-$C$-space if
for each sequence $\U_1, \U_2, \ldots$
of $m$-element open covers of $X$, there exists a sequence $\V_1, \V_2, \ldots$,
such that each $\V_i$ is a family of pairwise disjoint open subsets of $X$,
$V_i\prec\U_i$ and $\bigcup_{i=1}^\infty \V_i$ is a cover of $X$.
\end{defi}

Observe that
\begin{center}
$2$-$C$-spaces $\supseteq 3$-$C$-spaces $\supseteq\ldots\supseteq m$-$C$-spaces
$\supseteq\ldots\supseteq C$-spaces.
\end{center}
Moreover, the following holds
\begin{fact}[{\cite[Proposition 2.11]{Fed}}]
\label{wid}
A space is weakly infinite dimensional if and only if it is a $2$-$C$-space.
\end{fact}

One can  easily adopt the proof of Theorem \ref{relf-C}
to obtain the following:
\begin{pro}
\label{m-C-refl}
Let $K^*$ be a lattice in $\M$ and $K=K^*\cap\M$.
Then $wK^*$ is an $m$-$C$-space if and only if such is $wK$.\edowod
\end{pro}

Let us recall two definitions and one question from \cite{vdS}:
\begin{defi}[{\cite[Definition 2.7]{vdS}}]
We will say that a property $\mathcal{P}$ of a compact space is \emph{elementarily reflected}
if whenever some compact space $X$ has the property $\mathcal{P}$ then the Wallman representation
$wL$ of any elementary sublattice $L$ of $2^X$ also has property $\mathcal{P}$.
\end{defi}
\begin{defi}[{\cite[Definition 2.8]{vdS}}]
A property $\mathcal{P}$ of a compact space is \emph{elementarily reflected by submodels}
if whenever some compact space $X$ has the property $\mathcal{P}$ then the Wallman representation
$wL$ of any elementary sublattice $L$ of the form $L=2^X\cap\M$, where $2^X\in\M$ and $\M\prec H(\kappa)$
(for a large enough regular $\kappa$),
 also has property $\mathcal{P}$.
\end{defi}
\begin{que}[{\cite[Question 2.32]{vdS}}]
Is having strong infinite dimension elementarily reflected, and is having not
strong infinite dimension elementarily reflected?
\end{que}
Recall that, by  definition, a space is strongly infinite dimensional if it
is not weakly  infinite dimensional. 

Proposition \ref{m-C-refl} gives a partial answer to this question.
Indeed, in particular it says that that both these properties are elementarily reflected by submodels
(use the characterization of the weak infinite dimension from Fact \ref{wid}).
Moreover, following the proof of the  Theorem \ref{relf-C}
one can observe that the model $\M\prec H(\kappa)$ is not needed for the left-to-right implication.
That means property $C$ is elementarily reflected and the opposite is elementarily
reflected by submodels. Properties \mbox{$m$-$C$} and non-$m$-$C$ behave in the same way.
Summarizing, we can say that having strong infinite dimension is elementarily reflected by submodels,
and having not strong infinite dimension is elementarily reflected.

It is not known if the notions of property $C$ and weakly infinite dimension coincide
within the class of compact spaces.
However, since both properties are elementarily reflected by submodels,
there exists a metric counterexample which distinguishes these two notions if and only if
there exists a non-metric one.

\begin{acknowledgement}
The author is indebted to Piotr Borodulin-Nadzieja, Ahmad Farhat and Pawe\l\ Krupski
for helpful discussions.
\end{acknowledgement}

\bibliographystyle{amsplain}

\begin{thebibliography}{99}
\bibitem {BHHS} D. Barto\v sov\' a, K. P. Hart, L. C. Hoehn, B. van der Steeg,
\textit{Lelek's problem is not a metric problem},
Topol. Appl. 158 (2011), 2479-2484.
\bibitem {EbNa} C. Eberhart, S. B. Nadler, \textit{The Dimension of Certain Hyperspaces},
Bull. Acad. Polon. Sci. Ser. Sci. Math. Astronom. Phys. 19 (1971), 1027-1034.
\bibitem {EngD} R. Engelking, \textit{Theory of Dimensions Finite and Infinite},
Heldermann Verlag, 1995.
\bibitem {Fed} V. V. Fedorchuk, \textit{Some classes of weakly infinite-dimensional spaces},
J. Math. Sci. (N. Y.) 155 (2008), no.4., 523-570.
\bibitem {Hart} K. P. Hart, \textit{Elementarity and Dimensions},
Math. Notes, vol. 78, no. 2, 2005, 264-269.
\bibitem {HartEPol} K. P. Hart, E. Pol, \textit{On hereditarily indecomposable
compacta and factorization of maps}, Houston J. Math., 37 (2011), 637-644.
\bibitem {Jech} T. Jech, \textit{Set Theory}, The Third Millenium Edition, Springer.
\bibitem {LeRo} M. Levin, J. T. Rogers, Jr., \textit{A generalization of Kelley's theorem
for C-spaces}, Proc. Amer. Math. Soc., 128(1999), 1537-1541.
\bibitem {LeSt} M. Levin, Y. Sternfeld, \textit{The space of subcontinua of a 2-dimensional
continuum is infinite dimensional}, Proc. Amer. Math. Soc., 125(1997), 2771-2775.
\bibitem {Sta} W. Stadnicki, \textit{On the hyperspace dimension}, 2012,
article available at\\
http://ssdnm.mimuw.edu.pl/pliki/prace-studentow/st/pliki/wojciech-stadnicki-3.pdf
\bibitem {vdS} B.J. van der Steeg, \textit{Models in Topology}, DUP Science, 2003.






\end{thebibliography}

\end{document}